\documentclass{amsart}
\usepackage{amsfonts}
\usepackage{amsmath}
\usepackage{amssymb}
\usepackage{amsthm}

\newtheorem{theorem}{Theorem}
\newtheorem{prop}{Proposition}
\newtheorem{lemma}{Lemma}
\newtheorem{rem}{Remark}
\newtheorem{exmp}{Example}

\begin{document}

\title[Wigner's type theorem in terms of linear operators]
{Wigner's type theorem in terms of linear operators which send projections of a fixed rank to projections of other fixed rank}
\author{Mark Pankov}
\subjclass[2000]{}
\address{Faculty of Mathematics and Computer Science, 
University of Warmia and Mazury, S{\l}oneczna 54, Olsztyn, Poland}
\email{pankov@matman.uwm.edu.pl}

\maketitle

\begin{abstract}
Let $H$ be a complex Hilbert space whose dimension is not less than $3$
and let ${\mathcal F}_{s}(H)$ be the real vector space formed by all self-adjoint operators of finite rank on $H$.
For every non-zero natural $k<\dim H$ we denote by ${\mathcal P}_{k}(H)$ the set of all rank $k$ projections.
Let $H'$ be other complex Hilbert space of dimension not less than $3$
and let $L:{\mathcal F}_{s}(H)\to {\mathcal F}_{s}(H')$ be a linear operator such that
$L({\mathcal P}_{k}(H))\subset {\mathcal P}_{m}(H')$ for some natural $k,m$ and the restriction of $L$ to ${\mathcal P}_{k}(H)$ is injective.
If $H=H'$ and $k=m$, then $L$ is induced by a linear or conjugate-linear isometry of $H$ to itself, except the case $\dim H=2k$
when there is another one possibility (we get a classical Wigner's theorem if $k=m=1$).
If $\dim H\ge 2k$, then $k\le m$.
The main result describes all linear operators $L$ satisfying the above conditions
under the assumptions that $H$ is infinite-dimensional and 
for any $P,Q\in {\mathcal P}_{k}(H)$ the dimension of the intersection of the images of $L(P)$ and $L(Q)$ is not less than $m-k$.
\end{abstract}

\section{Introduction}
By Gleason's theorem \cite{Gleason},
all (pure and mixed) states of  a quantum mechanical system can be identified with bounded self-adjoint non-negative operators of trace one
on a complex Hilbert space. 
Such operators form a convex set whose extreme points are rank one projections corresponding to pure states.  
Let $H$ be a complex Hilbert space of dimension not less than $3$.
For every non-zero natural $k<\dim H$ we denote by ${\mathcal P}_{k}(H)$
the set of all rank $k$ projections, i.e. self-adjoint idempotents whose images are $k$-dimensional subspaces.
Each ${\mathcal P}_{k}(H)$ can be identified with the class of states corresponding to operators $k^{-1}P$, where $P$ is a projection of rank $k$.

Classical Wigner's theorem says that all symmetries of the space of pure states are induced by unitary and anti-unitary operators.
There is also a non-bijective version of this result concerning transformations induced by linear and conjugate-linear isometries
and there are various extensions of Wigner's theorem on others ${\mathcal P}_{k}(H)$.
We show that a linear operator sending ${\mathcal P}_{k}(H)$ to ${\mathcal P}_{m}(H)$ 
can be obtained from a linear or conjugate-linear isometry if  the images of projections do not diverge far enough.

The non-bijective version of Wigner's theorem states that any (not necessarily bijective) transformation of ${\mathcal P}_{1}(H)$
preserving the trace of the composition of two projections (or equivalently, the angle between the corresponding rows)
is induced by a linear or conjugate-linear isometry \cite[Section 2.1]{Molnar-book}.
For $k\ge 2$ transformations of ${\mathcal P}_{k}(H)$ preserving the principal angles between the images of two projections 
and transformations preserving the trace of the composition of two projections are described in \cite[Section 2.1]{Molnar-book} and \cite{Geher},
respectively. 
All such transformations are induced by linear or conjugate-linear isometries,
except the case $\dim H=2k$ when there is an additional class generated by the transformation 
which replaces the projection on $X$ by the projection on the orthogonal complement $X^{\perp}$.
It was established in \cite{Pankov2} that
to get a transformation of ${\mathcal P}_{k}(H)$ induced by a  linear or conjugate-linear isometry it is sufficient to require that 
only some types of the principal angles (corresponding to the orthogonality and adjacency relations) are preserved.
There are also characterizations of unitary and anti-unitary operators in terms of orthogonality or compatibility preserving transformations
\cite{GeherSemrl,Gyory,Pankov1,Semrl}.

We present Wigner's type theorems of different nature.

Consider the real vector space ${\mathcal F}_{s}(H)$ formed by all finite rank self-adjoint operators on $H$.
Every such operator is a real linear combination of rank one projections (the spectral theorem)
and it is not difficult to show that each rank one projection can be presented as a real linear combination of 
rank $k$ projections for every natural $k\ge 2$. 
Therefore, the vector space ${\mathcal F}_{s}(H)$ is spanned by each ${\mathcal P}_{k}(H)$.
By \cite[Lemma 2.1.2]{Molnar-book},
if a transformation of ${\mathcal P}_{k}(H)$ preserves  the trace of the composition of two projections,
then it can be extended to an injective linear operator on ${\mathcal F}_{s}(H)$
(this fact also is exploited in \cite{Geher}).

Linear operators preserving projections of fixed finite rank were investigated in \cite{ACh,SChM,Stormer}. 
Let $L$ be a linear operator on ${\mathcal F}_{s}(H)$ such that 
$$L({\mathcal P}_{k}(H))\subset {\mathcal P}_{k}(H)$$ for certain natural $k$
and the restriction of $L$ to ${\mathcal P}_{k}(H)$ is injective.
This operator is induced by a linear or conjugate linear isometry if $\dim H\ne 2k$;
in the case when $\dim H=2k$, there is an additional class of operators satisfying the above conditions.
This statement is a small generalization of the result by Aniello and Chru\'sci\'nski \cite{ACh}.
In this paper, it will be presented as a simple consequence of \cite[Theorem 1]{Pankov2} and some arguments from \cite{Geher}.

Our main result concerns linear operators $L:{\mathcal F}_{s}(H)\to {\mathcal F}_{s}(H')$ 
(where $H'$ is other complex Hilbert space) satisfying the following conditions:
$$L({\mathcal P}_{k}(H))\subset {\mathcal P}_{m}(H')$$ 
for some natural $k,m$ and the restriction of $L$ to ${\mathcal P}_{k}(H)$ is injective.
In the case when $\dim H\ge 2k$, the existence of such operators implies that $k\le m$.
We determine all linear operators $L$ satisfying the above conditions 
under the assumptions that $H$ is infinite-dimensional and 
for any $P,Q\in {\mathcal P}_{k}(H)$ the dimension of the intersection of the images of $L(P)$ and $L(Q)$ is not less than $m-k$.
We will use a modification of the methods exploited to study isometric embeddings of Grassmann graphs \cite[Chapter 3]{Pankov-book}
and Geh\'er's characterization of the adjacency relation \cite[Lemma 2]{Geher}.

\section{Results}
Let $H$ and $H'$ be complex Hilbert spaces whose dimensions are not less than $3$. 
For every closed subspace $X$ we denote by $P_{X}$ the projection on $X$.
Consider a few examples of linear operators between ${\mathcal F}_{s}(H)$ and ${\mathcal F}_{s}(H')$
which send projections to projections.

\begin{exmp}{\rm
Let $U:H\to H'$ be a linear or conjugate-linear isometry. 
Then $U^{*}:H'\to H$ is surjective and $U^{*}U$ is identity.
For every $A\in{\mathcal F}_{s}(H)$ we define
$$L_{U}(A)=UAU^{*}$$
and get the injective linear operator $L_{U}:{\mathcal F}_{s}(H)\to {\mathcal F}_{s}(H')$.
It sends $P_X$ to $P_{U(X)}$ and 
$$L_{U}({\mathcal P}_{k}(H))\subset {\mathcal P}_{k}(H)$$
for every $k$.
This operator is invertible only in the case when $U$ is a unitary or anti-unitary operator.
If $\dim H=\dim H'$ is finite, then every linear or conjugate-linear isometry of $H$ to $H'$ is a unitary or anti-unitary operator.
}\end{exmp}

\begin{exmp}{\rm
Suppose that $\dim H=n$ is finite and fix non-zero natural $k<n$.
Consider the linear operator $L^{\perp}_{k}$ on ${\mathcal F}_{s}(H)$
defined as 
$$L^{\perp}_{k}(A)=k^{-1}{\rm tr }(A){\rm Id}_{H}-A$$
for every $A\in{\mathcal F}_{s}(H)$ (${\rm tr }(A)$ is the trace of $A$).
If $X$ is a $k$-dimensional subspace of $H$, then $L^{\perp}_{k}$ sends $P_{X}$ to $P_{X^{\perp}}$
and we have
$$L^{\perp}_{k}({\mathcal P}_{k}(H))={\mathcal P}_{n-k}(H),$$
i.e. $L^{\perp}_{k}$ preserves ${\mathcal P}_{k}(H)$ if $n=2k$.
The operator $L^{\perp}_{k}$ is invertible.
}\end{exmp}

\begin{exmp}{\rm
Let $k$ and $m$ be natural numbers such that $k<\dim H$, $m<\dim H'$ and $k\le m$. 
Consider an $(m-k)$-dimensional subspace $W\subset H'$ and
denote by $H''$ the orthogonal complement of $W$.
Suppose that $U:H\to H''$ is a  linear or conjugate-linear isometry. 
For every $A\in {\mathcal F}_{s}(H)$ we denote by $L_{U,W}(A)$ the element of ${\mathcal F}_{s}(H')$
whose restriction to $H''$ coincides with $L_{U}(A)$ and whose restriction to $W$ is
$k^{-1}{\rm tr}(A){\rm Id}_{W}$;
in other words,
$$L_{U,W}(A)=L_{U}(A)P_{H'}+k^{-1}{\rm tr}(A)P_{W}.$$
We get the injective linear operator $L_{U,W}:{\mathcal F}_{s}(H)\to {\mathcal F}_{s}(H')$
(which coincides with $L_{U}$ if $W=0$) and 
$$L_{U,W}({\mathcal P}_{k}(H))\subset {\mathcal P}_{m}(H').$$
The image of the projection on a $k$-dimensional subspace $X$ is the projection on 
the $m$-dimensional subspace $U(X)+W$.
}\end{exmp}

Our first statement is a generalization of  Aniello--Chru\'sci\'n\-ski's result \cite{ACh}.

\begin{theorem}\label{theorem1}
Let $L$ be a linear operator on ${\mathcal F}_{s}(H)$.
Suppose that there is natural $k$ such that 
$$L({\mathcal P}_{k}(H))\subset {\mathcal P}_{k}(H)$$
and the restriction of $L$ to ${\mathcal P}_{k}(H)$ is injective.
Then  one of the following possibilities is realized:
\begin{enumerate}
\item[$\bullet$] $L=L_{U}$ for a certain linear or conjugate-linear isometry $U$,
\item[$\bullet$] $\dim H=2k$ and $L=L^{\perp}_{k}L_{U}$, where $U$ is a unitary or anti-unitary operator.
\end{enumerate}
\end{theorem}

The general case can be reduced to the case when $\dim H\ge 2k$. 
If $\dim H=n$ is finite and $k>n-k$, then  we consider the operator $L^{\perp}_{k}LL^{\perp}_{n-k}$ 
which sends projections on $(n-k)$-dimensional subspaces to projections on $(n-k)$-dimensional subspaces
and its restriction to ${\mathcal P}_{n-k}(H)$ is injective.
Suppose that 
$$L^{\perp}_{k}LL^{\perp}_{n-k}=L_{U}$$
for a certain unitary or anti-unitary operator $U$. 
We have $U(X^{\perp})=U(X)^{\perp}$ for every closed subspace $X\subset H$
which guarantees that $L$ maps the projection on a $k$-dimensional subspace $X$ to the projection on $U(X)$.
So, $L(P)=L_{U}(P)$ for all $P\in {\mathcal P}_{k}(H)$ which means that 
$L=L_{U}$, since ${\mathcal F}_{s}(H)$ is spanned by ${\mathcal P}_{k}(H)$.

In the case when $\dim H\ge 2k$, this statement is a consequence of the results from the next section.

\begin{exmp}{\rm
Let $P$ be a projection of rank $k$. 
Consider the  linear operator which sends every $A\in {\mathcal F}_{s}(H)$ to $k^{-1}{\rm tr}(A)P$. 
It transfers every projection of rank $k$ to $P$,
i.e. the assumption that the restriction of $L$ to ${\mathcal P}_{k}(H)$ is injective cannot be omitted.
}\end{exmp}

Our main result is the following.

\begin{theorem}\label{theorem2}
Suppose that $L:{\mathcal F}_{s}(H)\to {\mathcal F}_{s}(H')$ is a linear operator satisfying the following conditions:
\begin{enumerate}
\item[(L1)] $L({\mathcal P}_{k}(H))\subset {\mathcal P}_{m}(H')$ for some natural $k$ and $m$,
\item[(L2)]  the restriction of $L$ to ${\mathcal P}_{k}(H)$ is injective.
\end{enumerate}
If $\dim H\ge 2k$, then $k\le m$. 
Suppose that $H$ is infinite-dimensional and the following condition holds:
\begin{enumerate}
\item[(L3)] for any $P,Q\in {\mathcal P}_{k}(H)$
the dimension of the intersection of the images of $L(P)$ and $L(Q)$ is not less than $m-k$.
\end{enumerate}
Then $L=L_{U,W}$, where $W$ is an $(m-k)$-dimensional subspace of $H'$ and 
$U$ is a linear or conjugate-linear isometry of $H$ to the orthogonal complement of $W$.
\end{theorem}

\section{Preliminary}
\subsection{Semilinear maps}
Let $V$ and $V'$ be left vector spaces over division rings $R$ and $R'$, respectively. 
The dimensions of the vector spaces  are assumed to be not less than $3$.
Denote by ${\mathcal G}_{k}(V)$ the Grassmannian formed by $k$-dimensional subspaces of $V$.
Recall that a {\it line} of ${\mathcal G}_{1}(V)$ is the set of all $1$-dimensional subspaces contained in a certain $2$-dimensional subspace of $V$.
A map $L:V\to V'$ is called {\it semilinear} if 
$$L(x+y)=L(x)+L(y)$$
for all $x,y\in V$ and there is a ring homomorphism $\sigma:R\to R'$ such that 
$$L(ax)=\sigma(a)L(x)$$
for all $x\in V$ and $a\in R$. 
Every semilinear injection $L:V\to V'$ induces a map of ${\mathcal G}_{1}(V)$ to ${\mathcal G}_{1}(V')$
which sends lines to subsets of lines.
This map is not necessarily injective 
(if the associated ring homomorphism is not surjective, 
then $L$ can transfer distinct $1$-dimensional subspaces of $V$ to subsets of the same $1$-dimensional subspace of $V'$).
We will need the following version of the Fundamental Theorem of Projective Geometry.

\begin{theorem}[Faure and Fr\"olicher \cite{FF}, Havlicek \cite{Havlicek}]\label{FTPG}
Let $f:{\mathcal G}_{1}(V)\to {\mathcal G}_{1}(V')$ be an injection satisfying the following conditions:
\begin{enumerate}
\item[(1)] $f$ sends lines to subsets of lines,
\item[(2)] the image $f({\mathcal G}_{1}(V))$ is not contained in a line.
\end{enumerate}
Then $f$ is induced by a semilinear injection of $V$ to $V'$.
\end{theorem}

\begin{rem}{\rm
For non-injective $f:{\mathcal G}_{1}(V)\to {\mathcal G}_{1}(V')$ satisfying (1) and (2)
the statement fails \cite[Example 2.3]{Pankov-book}.
}\end{rem}

Let $X$ be  a set and let $R\subset X\times X$ be a symmetric relation on $X$.
A transformation $f:X\to X$ is said to be $R$ {\it preserving}  if 
$$(x,y)\in R\;\Longrightarrow (f(x),f(y))\in R;$$
in the case when 
$$(x,y)\in R\;\Longleftrightarrow (f(x),f(y))\in R,$$
we say that $f$ is $R$ {\it preserving in both directions}.

\begin{lemma}\label{lemma0-1}
Every  semilinear injection of $H$ to $H'$ preserving the orthogonality relation is a linear or conjugate-linear isometry. 
\end{lemma}

\begin{proof}
See, for example, \cite[Lemma 4]{Pankov2}.
\end{proof}

\subsection{Grassmann graph}
For subspaces $X,Y\subset H$ satisfying $\dim X < k <\dim Y$ and $X\subset Y$
we denote by $[X,Y]_{k}$ the set of all $k$-dimensional subspaces $Z$ such that $X\subset Z\subset Y$.
In the cases when $X=0$ and $Y=H$, we will write $\langle Y]_{k}$ and $[X\rangle_{k}$, respectively.

Suppose that $\dim H\ge 2k$ and consider the Grassmann graph $\Gamma_{k}(H)$
whose vertex set is ${\mathcal G}_{k}(H)$ and two distinct $k$-dimensional subspaces are adjacent vertices
if their intersection is $(k-1)$-dimensional. 
If $k=1$, then any two distinct vertices in this graph are adjacent. 
In the case when $k\ge 2$, there is the following description of maximal cliques of $\Gamma_{k}(H)$
(a clique in a graph is a subset of the vertex set consisting of mutually adjacent vertices):
every maximal clique is a {\it star} $[X\rangle_{k}$, $X\in {\mathcal G}_{k-1}(H)$
or a {\it top} $\langle Y]_{k}$, $Y\in {\mathcal G}_{k+1}(H)$.

The {\it distance} between two vertices in a connected graph is the smallest number $i$ such that 
there is a path consisting of $i$ edges and connecting these vertices.
A path connecting vertices $v$ and $w$ is a {\it geodesic} if the number of vertices in this path is equal to the distance between $v$ and $w$.
The graph $\Gamma_{k}(V)$ is connected and the distance between $X,Y\in{\mathcal G}_{k}(H)$ in this graph will be denoted by $d_{k}(X,Y)$. 
We have
$$d_{k}(X,Y)=k-\dim (X\cap Y)$$
for any $X,Y\in{\mathcal G}_{k}(H)$.

Two closed subspaces $X,Y\subset H$ are called {\it compatible} if there is an orthonormal basis of $H$ such that $X$ and $Y$
both are spanned by subsets of this basis. This is equivalent to the existence of closed orthogonal subspaces 
$X'\subset X$ and $Y'\subset Y$ such that $X\cap Y$ is orthogonal to both $X',Y'$ and 
$$X=X'+(X\cap Y),\;\;\;Y=Y'+(X\cap Y).$$
A subset of ${\mathcal G}_{k}(H)$ is said to be {\it compatible} if any two elements from this subset are compatible.
We say that two elements of ${\mathcal G}_{k}(H)$ are {\it ortho-adjacent} if they are compatible and adjacent.
Compatible subsets of stars and tops consist of mutually ortho-adjacent elements.

\begin{lemma}\label{lemma0-2}
Every maximal compatible subset of a top ${\mathcal T}\subset {\mathcal G}_{k}(H)$ consists of $k+1$ elements.
Every maximal compatible subset of a star ${\mathcal S}\subset {\mathcal G}_{k}(H)$ consists of $\dim H -k +1$ elements if 
$H$ is finite-dimensional, and it is infinite if $H$ is infinite-dimensional. 
\end{lemma}

\begin{proof}
Easy verification.
\end{proof}

\begin{lemma}\label{lemma0-3}
Every geodesic of $\Gamma_{k}(H)$ connecting compatible elements is formed by mutually compatible elements.
Any two compatible elements of ${\mathcal G}_{k}(H)$ are contained in a geodesic of $\Gamma_{k}(H)$ connecting orthogonal elements.
\end{lemma}

\begin{proof}
See \cite[Section 3]{Pankov2}.
\end{proof}

Let $X,Y\in {\mathcal G}_{k}(H)$. 
Denote by ${\mathcal X}_{k}(X,Y)$ the set of all 
$Z\in {\mathcal G}_{k}(H)$ such that $$P_{X}+P_{Y}-P_{Z}\in {\mathcal P}_{k}(H).$$ 
In other words, $Z\in {\mathcal G}_{k}(H)$ belongs to ${\mathcal X}_{k}(X,Y)$ if and only if 
there is $Z'\in {\mathcal G}_{k}(H)$ satisfying
$$P_{X}+P_{Y}=P_{Z}+P_{Z'}.$$
Since the image of $P_{X}+P_{Y}$ coincides with $X+Y$, the set ${\mathcal X}_{k}(X,Y)$ is contained in 
${\mathcal G}_{k}(X+Y)$. 
Also, every element of ${\mathcal X}_{k}(X,Y)$ contains $X\cap Y$
(this follows from the fact that $(P_{X}+P_{Y})(x)=2x$ if and only if $x\in X\cap Y$).
Therefore, 
$${\mathcal X}_{k}(X,Y)\subset [X\cap Y, X+Y]_{k}.$$
The converse inclusion holds only in the case when $X$ and $Y$ are compatible.
In particular, ${\mathcal X}_{k}(X,Y)$ coincides with ${\mathcal G}_{k}(X+Y)$ if and only if $X$ and $Y$ are orthogonal.

\begin{lemma}[Geh\'er \cite{Geher}]\label{lemma0-4}
The set ${\mathcal X}_{k}(X,Y)$ is a $1$-dimensional real manifold if and only if
$X,Y\in{\mathcal G}_{k}(H)$ are non-compatible and adjacent.
\end{lemma}

\subsection{Geometric version of Wigner's theorem for Hilbert Grassmannians}
Every linear or conjugate-linear isometry of $H$ to itself induces an injective transformation of ${\mathcal G}_{k}(H)$
preserving the adjacency and orthogonality relations in both directions
(note that pairs of orthogonal $k$-dimensional subspaces exist only in the case when $\dim H\ge 2k$).
If $\dim H=2k$, 
then the same holds for the orthocomplementation map which sends every $X\in {\mathcal G}_{k}(H)$ to $X^{\perp}\in {\mathcal G}_{k}(H)$.
To prove Theorem \ref{theorem1} we use the following.

\begin{theorem}[Pankov \cite{Pankov2}]\label{Pank-theorem}
If $\dim H>2k>2$, then every transformation of ${\mathcal G}_{k}(H)$ preserving the adjacency and orthogonality relations
is induced by a linear or conjugate-linear isometry.
\end{theorem} 

For $\dim H=2k$ there is the following weak version of the above result.

\begin{prop}[Pankov \cite{Pankov2}]\label{Pank-prop}
If $\dim H=2k>2$ and $f$ is an orthogonality preserving transformation of ${\mathcal G}_{k}(H)$
which also preserves the adjacency relation in both directions, 
then $f$ is induced by a unitary or anti-unitary operator or 
it is the composition of the orthocomplementation and the transformation induced by a unitary or anti-unitary operator.
\end{prop}

\section{Proof of Theorems 1 and 2}
Let $\dim H\ge 2k$ and let $L:{\mathcal F}_{s}(H)\to {\mathcal F}_{s}(H)$ be a linear operator satisfying the following conditions:
\begin{enumerate}
\item[(L1)] $L({\mathcal P}_{k}(H))\subset {\mathcal P}_{m}(H')$,
\item[(L2)]  the restriction of $L$ to ${\mathcal P}_{k}(H)$ is injective.
\end{enumerate}
Consider the injection $f:{\mathcal G}_{k}(H)\to {\mathcal G}_{m}(H')$ induced by $L$, i.e. 
$$L(P_{X})=P_{f(X)}$$
for every $X\in {\mathcal G}_{k}(H)$.
It is clear that for any $X,Y\in {\mathcal G}_{k}(H)$ we have 
$$f({\mathcal X}_{k}(X,Y))\subset {\mathcal X}_{m}(f(X),f(Y)).$$
In the case when $X,Y$ are orthogonal,
${\mathcal X}_{k}(X,Y)$ coincides with ${\mathcal G}_{k}(X+Y)$ and 
$$f({\mathcal G}_{k}(X+Y))\subset {\mathcal X}_{m}(f(X),f(Y))\subset [f(X)\cap f(Y), f(X)+f(Y)]_{m}.$$
This means that $L$ transfers ${\mathcal F}_{s}(X+Y)$ to a subspace of ${\mathcal F}_{s}(f(X)+f(Y))$\footnote{For every closed subspace
$Z\subset H$ the subspace of ${\mathcal F}_{s}(H)$ formed by all operators whose images are contained in $H$ can be naturally identified with
${\mathcal F}_{s}(Z)$.}.
Since  ${\mathcal F}_{s}(X+Y)$ and ${\mathcal F}_{s}(f(X)+f(Y))$ are finite-dimensional vector spaces and $f$ is induced by $L$,
the restriction of $f$ to ${\mathcal G}_{k}(X+Y)$ is continuous. 
By (L2) and the compactness of ${\mathcal G}_{k}(X+Y)$, this restriction is a homeomorphism on $f({\mathcal G}_{k}(X+Y))$.
Therefore, 
\begin{equation}\label{eq1-1}
[f(X)\cap f(Y), f(X)+f(Y)]_{m}
\end{equation}
contains a $(2k^{2})$-dimensional real manifold. This implies that $k\le m$.

From this moment, we assume that the following additional condition holds:
\begin{enumerate} 
\item[(L3)] for any $P,Q\in {\mathcal P}_{k}(H)$
the dimension of the intersection of the images of $L(P)$ and $L(Q)$ is not less than $m-k$.
\end{enumerate}
In the case when $k=m$, this condition holds trivially. 

\begin{lemma}\label{lemma1-1} 
The following assertions are fulfilled:
\begin{enumerate}
\item[(A)] if $X,Y\in {\mathcal G}_{k}(H)$ are orthogonal, then $f(X)$ and $f(Y)$ are compatible and 
$$\dim(f(X)\cap f(Y))=m-k.$$
\item[(B)] $f$ is adjacency preserving in both directions.
\end{enumerate}
\end{lemma}

\begin{proof}
(A).  Suppose that $X,Y\in {\mathcal G}_{k}(H)$ are orthogonal. 
It was established above that \eqref{eq1-1}
contains a $(2k^{2})$-dimensional real manifold and (L3) shows that this is possible only 
in the case when the dimension of $f(X)\cap f(Y)$ is equal to $m-k$.
Then  \eqref{eq1-1} is homeomorphic to ${\mathcal G}_{k}(X+Y)$ and 
$f({\mathcal G}_{k}(X+Y))$ is an open-closed subsets of \eqref{eq1-1}.
Therefore,
$$f({\mathcal G}_{k}(X+Y))=[f(X)\cap f(Y), f(X)+f(Y)]_{m}$$
and onsequently
$${\mathcal X}_{m}(f(X),f(Y))=[f(X)\cap f(Y), f(X)+f(Y)]_{m}$$
which means that $f(X)$ and $f(Y)$ are compatible.

(B).
For any $X',Y'\in {\mathcal G}_{k}(H)$ there are orthogonal $X,Y\in {\mathcal G}_{k}(H)$ whose sum contains both $X',Y'$. 
The restriction of $f$ to ${\mathcal G}_{k}(X+Y)$ is a homeomorphism to  \eqref{eq1-1}.
This implies that $L({\mathcal F}_{s}(X+Y))$ is the subspace of ${\mathcal F}_{s}(H')$ spanned by all $P_{Z}$,
where $Z$ belongs to \eqref{eq1-1}.
This subspace can be identified with ${\mathcal F}_{s}(\tilde{H})$, where $\tilde{H}$ is a complex Hilbert space of dimension $2k$.
The restriction of $L$ to ${\mathcal F}_{s}(X+Y)$ is a linear isomorphism of ${\mathcal F}_{s}(X+Y)$ to ${\mathcal F}_{s}(\tilde{H})$
(this is a surjective linear operator between vector spaces are of the same finite dimension).
Then
$$f({\mathcal X}_{k}(X',Y'))={\mathcal X}_{m}(f(X'),f(Y')).$$
By Lemma \ref{lemma0-4}, $X',Y'$ are non-compatible and adjacent if and only if the same holds for $f(X'),f(Y')$.
The required statement follows from the fact that the set of all elements of ${\mathcal G}_{k}(X+Y)$ adjacent to $X'$ is the closure of the set of all  
$Y'\in {\mathcal G}_{k}(X+Y)$ such that $X',Y'$ are non-compatible and adjacent.
\end{proof}

\begin{proof}[Proof of Theorem \ref{theorem1}]
Let $k=m$. 
By Lemma \ref{lemma1-1}, $f$ is orthogonality preserving and it is also adjacency preserving in both directions
(in the case when $k=1$, the second statement is trivial).

Consider the case when $k=1$. 
Let $S\in {\mathcal G}_{2}(H)$ and let $X,Y$ be orthogonal elements of ${\mathcal G}_{1}(S)$.
Then ${\mathcal X}_{1}(X,Y)$ coincides with ${\mathcal G}_{1}(S)$ and 
$$f({\mathcal G}_{1}(S))=f({\mathcal X}_{1}(X,Y))={\mathcal X}_{1}(f(X),f(Y))={\mathcal G}_{1}(f(X)+f(Y)).$$
So, $f$ sends lines to lines. 
By (L2), $f$ is injective.
The image $f({\mathcal G}_{1}(H))$ is not contained in a line.
It follows from Theorem \ref{FTPG} that $f$ is induced by a semilinear injective transformation of $H$.
This semilinear transformation is orthogonality preserving (since $f$ is orthogonality preserving)
and Lemma \ref{lemma0-1} implies that it  is a linear or conjugate-linear isometry, i.e.
there is a linear or conjugate-linear isometry $U$ such that $f(X)=U(X)$ for every $X\in {\mathcal G}_{1}(H)$. 
Then $L(P)=L_{U}(P)$ for every $P\in {\mathcal P}_{1}(H)$ which implies that 
$L=L_{U}$ (since ${\mathcal F}_{s}(H)$ is spanned by ${\mathcal P}_{1}(H)$).

Similarly, we show that for $k\ge 2$ the statement follows from Theorem \ref{Pank-theorem} and Proposition \ref{Pank-prop}.
\end{proof}

Suppose that $H$ is infinite-dimensional  and prove Theorem \ref{theorem2} in several steps.
The case when $k=m$ was considered above and we assume that $m>k$.

\begin{proof}[Proof of Theorem \ref{theorem2} for $k=1$]
If $k=1$, then $m\ge 2$.
The second part of Lemma \ref{lemma1-1} implies that any two distinct elements of $f({\mathcal G}_{1}(H))$ are adjacent,
i.e. $f({\mathcal G}_{1}(H))$ is contained in a star or a top of ${\mathcal G}_{m}(H')$.
By the first part of Lemma \ref{lemma1-1},
for any orthogonal $X,Y\in {\mathcal G}_{1}(H)$ the images $f(X),f(Y)$ are ortho-adjacent.
If ${\mathcal X}\subset {\mathcal G}_{1}(H)$ is an infinite subset  formed by mutually orthogonal elements
(such subsets exist, since $H$ is infinite-dimensional), 
then $f({\mathcal X})$ is an infinite subset of ${\mathcal G}_{m}(H')$ consisting of mutually ortho-adjacent elements.
By Lemma \ref{lemma0-2}, this means that $f({\mathcal G}_{1}(H))$ is not contained in a top.
So, there is an $(m-1)$-dimensional subspace $W$ contained in each element of $f({\mathcal G}_{1}(H))$.

Let $H''$ be the orthogonal complement of $W$. 
Consider the map $$g:{\mathcal G}_{1}(H)\to {\mathcal G}_{1}(H'')$$ 
such that $g(X)=f(X)\cap H''$ for every $X\in {\mathcal G}_{1}(H)$.
This map is orthogonality preserving.
Since for any orthogonal $X,Y\in{\mathcal G}_{1}(H)$ we have 
$$f({\mathcal G}_{1}(X+Y))=[W, f(X)+f(Y)]_{m}$$
(see the proof of Lemma \ref{lemma1-1}), 
the map $g$ sends lines to lines. 
Also, $g$ is injective (since $f$ is injective) and it is clear that 
the image $g({\mathcal G}_{1}(H))$ is not contained in a line of ${\mathcal G}_{1}(H'')$.
Then $g$ is induced by a semilinear injection $U:H\to H''$ (Theorem \ref{FTPG}).
Lemma \ref{lemma0-1} guarantees that $U$ is a linear or conjugate-linear isometry.
Therefore, $$f(X)=U(X)+W$$ for all $X\in {\mathcal G}_{1}(H)$
and we have $L(P)=L_{U,W}(P)$ for all $P\in {\mathcal P}_{1}(H)$. 
This implies that $L=L_{U,W}$, since ${\mathcal F}_{s}(H)$ is spanned by ${\mathcal P}_{1}(H)$.
\end{proof}

From this moment, we suppose that $k\ge 2$.

\begin{lemma}\label{lemma1-2}
The map $f$ is ortho-adjacency preserving. 
\end{lemma}

\begin{proof}
Let $X$ and $Y$ be ortho-adjacent elements of ${\mathcal G}_{k}(H)$.
Then $f(X),f(Y)$ are adjacent and we need to show that they are compatible. 
By the second part of Lemma \ref{lemma1-1}, $f$ sends every path in $\Gamma_{k}(H)$ to a path in $\Gamma_{m}(H')$.
Consider a geodesic in $\Gamma_{k}(H)$ containing $X,Y$ and joining some orthogonal $X',Y'\in {\mathcal G}_{k}(H)$
(see Lemma \ref{lemma0-3}). 
Since 
$$d_{k}(X',Y')=k=d_{m}(f(X'),f(Y'))$$
(by the first part of Lemma \ref{lemma1-1}), $f$ transfers this geodesic to a geodesic connecting $f(X')$ and $f(Y')$.
The subspaces $f(X')$ and $f(Y')$ are compatible (by the first part of Lemma \ref{lemma1-1}) and Lemma \ref{lemma0-3}
states that any geodesic connecting them is formed by mutually compatible elements. 
In particular, $f(X)$ and $f(Y)$ are compatible.
\end{proof}

Since $f$ is adjacency preserving in both directions, 
it transfers every maximal clique of $\Gamma_{k}(H)$
(a star or a top) to a subset in a maximal clique of $\Gamma_{m}(H')$;
moreover, distinct maximal cliques of $\Gamma_{k}(H)$ go to subsets of distinct maximal cliques of $\Gamma_{m}(H')$.
Lemmas \ref{lemma0-2} and \ref{lemma1-2} show that the image of a star cannot be contained in a top
(recall that $H$ is infinite-dimensional).
Therefore, $f$ transfers every star ${\mathcal S}\subset {\mathcal G}_{k}(H)$ to a subset of a star. 
The intersection of two distinct stars contains at most one element
which means that $f({\mathcal S})$ is contained in a unique star of ${\mathcal G}_{m}(H)$.
So, $f$ induces an injection $$f_{k-1}:{\mathcal G}_{k-1}(H)\to {\mathcal G}_{m-1}(H)$$
such that 
$$f([X\rangle_{k})\subset[f_{k-1}(X)\rangle_{m}$$
for every $X\in {\mathcal G}_{k-1}(H)$. Then 
$$f_{k-1}(\langle Y]_{k-1})\subset \langle f(Y)]_{m-1}$$
for every $Y\in {\mathcal G}_{k}(H)$.
The latter inclusion implies that $f_{k-1}$ is adjacency preserving
(since $f_{k-1}$ is injective and two distinct elements are adjacent if and only if there is a top containing them).

\begin{lemma}\label{lemma1-3}
If $X,Y\in {\mathcal G}_{k-1}(H)$ are orthogonal, then $f_{k-1}(X)$ and $f_{k-1}(Y)$ are compatible and 
$$\dim(f_{k-1}(X)\cap f_{k-1}(Y))=m-k.$$
Also, $f_{k-1}$ is ortho-adjacency preserving\footnote{This follows immediately from the previous statement if $k=2$.}.
\end{lemma}

\begin{proof}
Since $f_{k-1}$ is an adjacency preserving injection, it transfers every path of $\Gamma_{k-1}(H)$ to a path of $\Gamma_{m-1}(H')$
and we have
$$d_{m-1}(f_{k-1}(X),f_{k-1}(Y))\le d_{k-1}(X,Y)=k-1.$$
Therefore,
$$d_{m-1}(f_{k-1}(X),f_{k-1}(Y))=m-1-\dim(f_{k-1}(X)\cap f_{k-1}(Y))\le k-1$$
which implies that
$$\dim (f_{k-1}(X)\cap f_{k-1}(Y))\ge m-k.$$
Let us take orthogonal $X',Y'\in {\mathcal G}_{k}(H)$ such that $X\subset X'$ and $Y\subset Y'$.
Then $f_{k-1}(X)$ and $f_{k-1}(Y)$ are contained in $f(X')$ and $f(Y')$, respectively.
The intersection of $f_{k-1}(X)$ and $f_{k-1}(Y)$  is contained in the intersection of $f(X')$ and $f(Y')$
which means that
$$\dim(f_{k-1}(X)\cap f_{k-1}(Y))\le \dim(f(X')\cap f(Y'))=m-k.$$
So, the dimension of the intersection of $f_{k-1}(X)$ and $f_{k-1}(Y)$ is equal to $m-k$ and 
\begin{equation}\label{eq2}
f_{k-1}(X)\cap f_{k-1}(Y)=f(X')\cap f(Y').
\end{equation}
The subspaces $f(X')$ and $f(Y')$ are compatible, i.e. 
there exist orthogonal $k$-dimen\-sional  subspaces $X''\subset f(X')$ and $Y''\subset f(Y')$
such that $f(X')\cap f(Y')$ is orthogonal to both $X'',Y''$ and
$$f(X')=X''+f(X')\cap f(Y'),\;\;\;f(Y')=Y''+f(X')\cap f(Y').$$
Since $f_{k-1}(X)$ and $f_{k-1}(Y)$ are hyperplanes in $f(X')$ and $f(Y')$ (respectively),
$$X''\cap f_{k-1}(X)\;\mbox{ and }\;Y''\cap f_{k-1}(Y)$$
are orthogonal $(k-1)$-dimensional subspaces which are also orthogonal to \eqref{eq2}. We have
$$f_{k-1}(X)=X''\cap f_{k-1}(X)+f_{k-1}(X)\cap f_{k-1}(Y),$$
$$f_{k-1}(Y)=Y''\cap f_{k-1}(Y)+f_{k-1}(X)\cap f_{k-1}(Y)$$
which means that $f_{k-1}(X)$ and $f_{k-1}(Y)$ are compatible.

As in the proof of Lemma \ref{lemma1-2}, we show that $f_{k-1}$ is ortho-adjacency preserving.
\end{proof}

Consider the case when $k\ge 3$.
Using the second part of Lemma \ref{lemma1-3}, 
we establish that for every star ${\mathcal S}\subset {\mathcal G}_{k-1}(H)$ there is a unique star of ${\mathcal G}_{m-1}(H')$
containing $f_{k-1}({\mathcal S})$.
This means that $f_{k-1}$ induces a map 
$$f_{k-2}:{\mathcal G}_{k-2}(H)\to {\mathcal G}_{m-2}(H')$$
which transfers tops to subsets of tops.
Therefore, for any adjacent $X,Y\in {\mathcal G}_{k-2}(H)$ the images $f_{k-2}(X),f_{k-2}(Y)$ are adjacent or coincident
(we cannot state that $f_{k-2}$ is injective). 
Then $f_{k-2}$ sends 
any path of $\Gamma_{k-2}(H)$ to a path of $\Gamma_{m-2}(H)$.
We repeat the above arguments and get the direct analogue of Lemma \ref{lemma1-3} for $f_{k-2}$.

Recursively, we construct a sequence of maps
$$f_{i}:{\mathcal G}_{i}(H)\to {\mathcal G}_{m-k+i}(H'),$$
where $f_{k}=f$ and 
$$f_{i}([X\rangle_{i})\subset[f_{i-1}(X)\rangle_{m-k+i}$$
for every $X\in {\mathcal G}_{i-1}(H)$ if $i\ge 2$.
Then 
$$f_{i}(\langle Y]_{i})\subset \langle f_{i+1}(Y)]_{m-k+i}$$
for every $Y\in {\mathcal G}_{i+1}(H)$ if $i\le k-1$.
The direct analogue of Lemma \ref{lemma1-3} holds for each $f_{i}$,
but we will need only the fact that 
$f_{1}$ sends orthogonal elements of ${\mathcal G}_{1}(H)$ to  ortho-adjacent elements of ${\mathcal G}_{m-k+1}(H')$.

For any distinct $X,Y\in {\mathcal G}_{1}(H)$ the images $f_{1}(X),f_{1}(Y)$ are adjacent or coincident 
(since $f_{1}$ is induced by $f_{2}$) and $f_{1}({\mathcal G}_{1}(H))$ is contained in a maximal clique of $\Gamma_{m-k+1}(H')$.
Using the fact that $f_{1}$ transfers orthogonal elements of ${\mathcal G}_{1}(H)$ to  ortho-adjacent elements of ${\mathcal G}_{m-k+1}(H')$,
we establish the existence of an $(m-k)$-dimensional subspace $W$ contained in each element of $f_{1}({\mathcal G}_{1}(H))$
(see the proof of Theorem \ref{theorem2} for $k=1$).

We will use the following inclusion
\begin{equation}\label{eq3}
f_{1}(\langle X]_{1})\subset \langle f(X)]_{m-k+1}\;\mbox{ for all }\;X\in{\mathcal G}_{k}(H)
\end{equation}
which is a simple consequence of the fact that $f_{i}$ is induced by $f_{i+1}$ for each $i\le k-1$.

\begin{lemma}\label{lemma1-4}
The map $f_{1}$ is injective.
\end{lemma}

\begin{proof}
For any distinct $X,Y\in {\mathcal G}_{1}(H)$ we take orthogonal $X_{1},\dots,X_{k-1}\in {\mathcal G}_{1}(H)$
which are orthogonal to both $X$ and $Y$. 
Then 
$$X'=X_{1}+\dots+X_{k-1}+X\;\mbox{ and }\;Y'=X_{1}+\dots+X_{k-1}+Y$$
are distinct $k$-dimensional subspaces. 
The $(m-k+1)$-dimensional subspaces 
$$f_{1}(X_{1}),\dots,f_{1}(X_{k-1}),f_{1}(X)$$
are mutually ortho-adjacent and each of them contains the $(m-k)$-dimensional subspace $W$.
This means that their sum is $m$-dimensional. 
Similarly, we establish that the subspace 
$$f_{1}(X_{1})+\dots+f_{1}(X_{k-1})+f_{1}(Y)$$
is $m$-dimensional.
The inclusion \eqref{eq3} shows that 
$$f(X')=f_{1}(X_{1})+\dots+f_{1}(X_{k-1})+f_{1}(X),$$
$$f(Y')=f_{1}(X_{1})+\dots+f_{1}(X_{k-1})+f_{1}(Y).$$
If $f_{1}(X)=f_{1}(Y)$, then we have $f(X')=f(Y')$ which contradicts the fact that $f$ is injective.
Therefore, $f_{1}(X)\ne f_{1}(Y)$. 
\end{proof}

As the proof of Theorem \ref{theorem2} for $k=1$,
we denote by $H''$ the orthogonal complement of $W$ and consider the map $g:{\mathcal G}_{1}(H)\to {\mathcal G}_{1}(H'')$
such that $g(X)=f_{1}(X)\cap H''$ for every $X\in {\mathcal G}_{1}(H)$.
It is clear that $g$  is orthogonality preserving and the image $g({\mathcal G}_{1}(H))$ is not contained in a line of ${\mathcal G}_{1}(H'')$.
By Lemma \ref{lemma1-4}, this map is injective.
Also, it sends lines to subsets of lines (since $f_{1}$ is induced by $f_{2}$).
So, $g$ satisfies the conditions of Theorem \ref{FTPG}.
Then there is a linear or conjugate-linear isometry $U:H\to H''$ such that 
$$f_{1}(X)=U(X)+W$$
for all $X\in {\mathcal G}_{1}(H)$
(see the proof of Theorem \ref{theorem2} for $k=1$).
Using \eqref{eq3}, we establish that $f(X)$ coincides with $U(X)+W$ for each $X\in {\mathcal G}_{k}(H)$,
i.e. $L(P)=L_{U,W}(P)$ for all $P\in {\mathcal P}_{k}(H)$. 
Since ${\mathcal F}_{s}(H)$ is spanned by ${\mathcal P}_{k}(H)$,
we have $L=L_{U,W}$.

\end{document}